\UseRawInputEncoding
\documentclass[11pt,reqno,a4paper]{amsart}

\usepackage{a4wide}
\usepackage{amssymb}
\usepackage{amsfonts}
\usepackage{enumerate}
\usepackage{amsmath}
\usepackage{bbm}
\usepackage{ stmaryrd }

\newtheorem{theorem}{Theorem}[section]

\newtheorem{definition}[theorem]{Definition}

\newtheorem{lemma}[theorem]{Lemma}

\newtheorem{remark}[theorem]{Remark}

\def\eps{\varepsilon}

\def\N{\mathbb{N}}

\def\P{\mathbb{P}}
\def\E{\mathbb{E}}

\DeclareMathOperator{\card}{card}
\DeclareMathOperator{\var}{var}

\newcommand{\x}{\bold{x}}
\newcommand{\y}{\bold{y}}
\newcommand{\M}{M^{RLE}}
\def\lt{\left}
\def\rt{\right}

\begin{document}

\title{R\'enyi entropy and pattern matching for run-length encoded sequences}
\author{J\'er\^ome Rousseau}
\address{J\'er\^ome Rousseau, Departamento de Matem\'atica, Universidade Federal da Bahia\\
Av. Ademar de Barros s/n, 40170-110 Salvador, Brazil}
\address{Departamento de Matem\'atica, Faculdade de Ci\^encias da Universidade do Porto,\\Rua do Campo Alegre, 687, 4169-007 Porto, Portugal}
\email{jerome.rousseau@ufba.br}
\urladdr{http://www.sd.mat.ufba.br/~jerome.rousseau}

\thanks{This work was partially supported by CNPq, by FCT project PTDC/MAT-PUR/28177/2017, with national funds, and by CMUP (UIDB/00144/2020), which is funded by FCT with national (MCTES) and European structural funds through the programs FEDER, under the partnership agreement PT2020.}

\maketitle
 
\begin{abstract}
In this note, we studied the asymptotic behaviour of the length of the longest common substring for run-length encoded sequences. When the original sequences are generated by an $\alpha$-mixing process with exponential decay (or $\psi$-mixing with polynomial decay), we proved that this length grows logarithmically with a coefficient depending on the R\'enyi entropy of the pushforward measure. For Bernoulli processes and Markov chains, this coefficient is computed explicitly.
\end{abstract}

\section{Introduction}
Since Big Data seems to be the trending field of (at least) this decade, data compression algorithms have become a fundamental tool for data storage and are in the first lines of the battle between storage costs, computations costs and delays in data availability. For an introduction to data compression we refer the reader to \cite{sayood} and to the unavoidable Lempel-Ziv algorithms \cite{LZ77,LZ78}.

For sequences with long runs of the same value, Run-Length Encoding (RLE) is a simple and efficient lossless data compression method. More precisely, for a run of the same value, the algorithm stored the value and the length of the run. For example, the following binary sequence
\[00001110000000011001111111111100000000\]
will be compressed as
\[(0,4)(1,3)(0,8)(1,2)(0,2)(1,9)(0,8).\]
Thus, this sequence of 37 characters will be represented after compression by a sequence of 14 characters.

RLE is typically used for image compression but has also application in image analysis \cite{hinds}, texture analysis of volumetric data \cite{texture} and has also been used for data compression of television signals \cite{tele} and fax transmission \cite{fax}.

Since pattern (or string) matching problems are not only highly significant in computer sciences, information theory and probability (see e.g. \cite{KS, KASW,deb,abgara, Neu, ColletRedig}) but also in biology \cite{W-book}, geology \cite{geo} and linguistics (e.g. \cite{deb2} and references therein) among others, algorithms to solve string matching problems for RLE strings have been developed (see e.g. \cite{apostolico,freschi,hoosh,ahsan,chen} and references within).
  
In this note, we will focus on a particular string matching problem: the longest common substring problem (or longest consecutive common subsequence problem). More precisely, we will concentrate on the asymptotics of the length of the longest common substring, i.e. for two sequences $\x$ and $\y$ drawn randomly from the same alphabet, the behaviour of
$$
M_n(\x,\y)= \displaystyle \max \left\{k \ : \x_i^{i+k-1} = \y_j^{j+k-1} \ \mbox{for some} \ 0 \leq i,j \leq n-k\right\}
$$
when $n\rightarrow\infty$.

In \cite{BaLiRo}, it was proved that for $\alpha$-mixing process with exponential decay $M_n \sim \frac{2}{H_2(\mu)}\log n$ almost surely, where $H_2(\mu)$ is the R\'enyi entropy (see Definition~\ref{defren}) of the stationary measure~$\mu$. Similar results have been proved for more than two sequences \cite{barrosrousseau} and for random sequences in random environment \cite{LCS-random}.

In \cite{CoLaRo}, the authors wondered if the above mentioned result holds if the sequences are transformed following certain rules of modification. Thus, if $f$ is a measurable function (called an encoder) transforming a sequence $\x$ into another sequence $f(\x)$, they studied the behaviour of $M_n(f(\x),f(\y))$ and obtain a relation with the R\'enyi entropy of the pushforward measure $f_*\mu$.

A natural question would be to ask if we could apply the results presented in \cite{CoLaRo} when the encoder is a compression algorithm and in particular the run-length encoder. Unfortunately, to obtain their main result, the authors needed that the encoder does not compress too much the sequences, an hypothesis which is not satisfied by the run-length encoder. Thus, we present here a different proof which allows us to prove, in Theorem~\ref{thprinc}, that, when $f$ is the run-length encoder and the original sequences are generated by an $\alpha$-mixing process with exponential decay (or $\psi$-mixing with polynomial decay), almost surely 
\[M_n(f(\x),f(\y)) \underset{n\rightarrow\infty}\sim \frac{2}{H_2(f_*\mu)}\log n.\]
We apply this result to Bernoulli processes (Example~\ref{exbern}) and Markov chains (Examples~\ref{exmark1} and~\ref{exmark2}), and, in these cases, compute explicitly $H_2(f_*\mu)$. We emphasize that for Markov chains the computation are different whether there are two or more than two states.

Other examples of processes satisfying our mixing assumptions are Gibbs states of a H\"older-continuous potential \cite{Bowen,Ruelle}, ARMA processes \cite{mokkadem}, some renewal processes \cite{abadi-car-gallo} and stationary determinantal process on the integer lattice \cite{FLQ}. We refer the reader to \cite{dou,brad1} for more examples and deep surveys on strong mixing conditions.

\section{Longest common substring for RLE sequences}
 We consider a stationary stochastic process $X=(X_n)_{n\in \N}$ over a finite alphabet $\mathcal{A}$, with stationary measure $\mu$. We will denote $\sigma$ the left shift and, for $i\in \N$,  $\sigma^i X=(X_{i+n})_{n\in \N}$. For $k\in\N$, we denote by $\mathcal{A}^k$ the set of cylinders or strings of length $k$ and the length of a cylinder $\omega$ will be denoted $|\omega|$. When there is no ambiguity, cylinders of $\mathcal{A}^k$ will be denoted $\omega$. We will use the notation $x_i^{i+k-1}$ if we need to indicate its time of occurrence $i$. Moreover, $\mu(\omega)$ will denote the probability $\mu(X_i^{i+k-1}=\omega)$ (which is independent of $i$ by stationarity). 
 
 We will be interested in some statistical properties of run-length encoded (RLE) sequences where the original sequences are generated by the stochastic process $X$.
  
 \begin{definition}
 Let $\mathcal{B}=\{(\alpha,k)\}_{\alpha\in\mathcal{A},k\in\N}$. We define the run-length encoder $f:\mathcal{A}^\N\rightarrow\mathcal{B}^\N$ by
 \[f(\underset{k_1}{\underbrace{\alpha_1\dots \alpha_1}}\underset{k_2}{\underbrace{\alpha_2\dots \alpha_2}}\dots\underset{k_n}{\underbrace{\alpha_n\dots \alpha_n}}\dots)=(\alpha_1,k_1)(\alpha_2,k_2)\dots(\alpha_n,k_n)\dots\]
 We observe that for all $i\in\N$, we consider that $\alpha_{i+1}\neq\alpha_i$.
  \end{definition}
 
 We will focus our analysis on the length of the longest common substring of RLE sequences:
 
 \begin{definition}
Given two sequences $\x,\y,$ we define the $n$-length of the longest common substring by
$$
M_n(\x,\y)= \displaystyle \max \left\{k \ : \x_i^{i+k-1} = \y_j^{j+k-1} \ \mbox{for some} \ 0 \leq i,j \leq n-k\right\}.
$$
and we will study the behaviour of the $n$-length of the longest common substring of the RLE sequences $f(\x),f(\y)$
$$
\M_n(\x,\y):= M_n(f(\x),f(\y)).
$$
\end{definition}
 We will prove that $\M_n$ is linked with the R\'enyi entropy of the pushforward measure $f_*\mu$. We recall that $f_*\mu(.)=\mu(f^{-1}(.))$ and we observe that $f_*\mu$ is the law of the stochastic process $f(X)$ but is in general not stationary. We give now the definition of R\'enyi entropy:
 \begin{definition}\label{defren}
For $k>1$, the lower and upper R\'enyi entropies of order $k$ of a measure $P$ are defined as
$$\underline{H}_k(P) = -\displaystyle\underset{n \to \infty}{\underline{\lim}}\frac{1}{(k-1)n}\log\sum\limits_{\omega} {P}(\omega)^k \ \  \mbox{and} \ \ \overline{H}_k({P}) = -\displaystyle\underset{n \to \infty}{\overline{\lim}}\frac{1}{(k-1)n}\log\sum\limits_{\omega} {P}(\omega)^k \ ,$$
where the sums are taken over all cylinders $\omega$ of length $n$. When the limit exists we denote by ${H}_k({P})$ the common value. We note that the R\'enyi entropy of order $k$ is also called generalized R\'enyi entropy and that the R\'enyi entropy of order $2$ is often only called R\'enyi entropy.
\end{definition}
The existence of the R\'enyi entropy of order $k$ has not been proved for general stochastic processes. However, it was proved for Bernoulli processes, finite Markov chains (e.g. \cite{rached}), infinite Markov chains \cite{ciuperca}, Gibbs measures of a H\"older-continuous potential \cite{haydn-vaienti}, for $\phi$-mixing measures \cite{luc}, for weakly $\psi$-mixing processes \cite{haydn-vaienti} and for $\psi_g$-regular processes \cite{AbCa}. 

\begin{definition}
The process $X$ with stationary measure $\mu$ is $\alpha$-mixing if there exists a function $\alpha: \mathbb{N} \to \mathbb{R}$ where $\alpha(g)$ converges to zero when $g$ goes to infinity and such that
\begin{equation*}
\sup_{A \in \mathcal{F}_0^n \ ; \  B \in \mathcal{F}_{n+g}^\infty}\lt|\mu\lt( A \cap B\rt) -\mu(A) \mu(B)\rt| \leq \alpha(g) \ ,
\end{equation*}
for all $n \in \mathbb{N}$, where  for $0\leq J\leq L\leq\infty$, $\mathcal{F}_J^L$ denotes the $\sigma$-algebra $\sigma (X_k, J\leq k\leq L)$.

 When $\alpha(g)$ decreases exponentially fast to zero, we say that the process is $\alpha$-mixing with exponential decay.

The process is $\psi$-mixing if there exists a function $\psi: \mathbb{N} \to \mathbb{R}$ where $\psi(g)$ converges to zero when $g$ goes to infinity and such that
\begin{equation*}
\sup_{A \in \mathcal{F}_0^n \ ; \  B \in \mathcal{F}_{n+g}^\infty}\lt|\frac{\mu\lt( A \cap B\rt) -\mu(A) \mu(B)}{\mu(A) \mu(B)}\rt| \leq \psi(g),
\end{equation*}
for all $n \in \mathbb{N}$. 
\end{definition}

To obtain information on the growth length of the longest common substring for RLE sequences, we will need an assumption on the decay of the measure of cylinders:

(A) There exist $c>0$ and $h>0$, such that for any $n\in\N$ and any $a\in\mathcal{A}$
\[\mu(\underset{n}{\underbrace{a\dots a}})\leq c e^{-hn}.\]

We observe that in particular this assumption is always satisfied if the process is $\psi$-mixing with summable decay \cite[Lemma 1]{GS}.

First of all, without mixing assumption, we will prove an upper bound for the growth rate of the length of the longest common substring for RLE sequences:

\begin{theorem}\label{thinf} If $\underline{H}_2(f_*\mu)>0$ and if hypothesis (A) is satisfied, then for almost every $\x,\y$,
\begin{eqnarray*}\label{ineq1discretecase}
\displaystyle\underset{n \to \infty}{\overline{\lim}}\frac{\M_n(\x,\y)}{\log n} \leq \frac{2}{\underline{H}_2(f_*\mu)} \cdot
\end{eqnarray*}
\end{theorem}

If the process is $\alpha$-mixing with an exponential decay (or $\psi$-mixing with polynomial decay) and if for cylinders $C_n$ of length $n$ in $\mathcal{B}^n$, their preimage $f^{-1}C_n$ is of length at most $h(n)$ with $h(n)={o}(n^\gamma)$ for some $\gamma>0$, one could use  the ideas of \cite{CoLaRo} to get a lower bound. Nevertheless, the run-length encoder does not satisfy this last necessary assumption since preimage of cylinders under $f$ can have arbitrary length. Thus we present a different proof here to obtain the lower bound.

\begin{theorem}\label{thprinc}
If $\underline{H}_2(f_*\mu)>0$, hypothesis (A) is satisfied and the process is $\alpha$-mixing with an exponential decay (or $\psi$-mixing with $\psi(g)=g^{-a}$ for some $a>0$),
then, for almost every realizations $\x,\y$,
\begin{eqnarray*}\label{ineq2discretecase}
\displaystyle\underset{n \to \infty}{\underline{\lim}}\frac{\M_n(\x,\y)}{\log n} \geq \frac{2}{\overline{H}_2(f_*\mu)} \cdot
\end{eqnarray*}
Thus, if the R\'enyi entropy exists, we get for almost every $\x,\y$,
\begin{equation*}
\lim_{n \to \infty}\frac{\M_n(\x,\y)}{\log n} = \frac{2}{{H}_2(f_*\mu)}.
\end{equation*}
\end{theorem}

 \begin{remark}[More than 2 sequences]
One could wonder what will happen if we want to study the growth rate of the length of the longest common substring for $k$ RLE sequences. Using the ideas presented in our proofs and in \cite[Section 4]{barrosrousseau}, one could prove that, under the same assumptions of Theorem~\ref{thprinc} and if the R\'enyi entropy of order $k$ exists and is strictly positive, for almost every realizations $\x_1,\x_2,\dots,\x_k$,
\begin{equation*}
\lim_{n \to \infty}\frac{\M_n(\x_1,\x_2,\dots,\x_k)}{\log n} = \frac{k}{(k-1){H}_k(f_*\mu)}.
\end{equation*}
 \end{remark}
 
Theorem~\ref{thinf} and Theorem~\ref{thprinc} will be proved in Section~\ref{secproof}. We will now give examples satisfying our assumptions and where the R\'enyi entropy of the pushforward measure can be explicitly computed. 
\section{Examples}
First of all, we will treat the case of Bernoulli processes and then of Markov chains. We emphasize that for Markov chains the situation and the computation are different when working with an alphabet of two symbols or an alphabet of more than two symbols.
 \subsection{Bernoulli process}\label{exbern}
 Let us consider the alphabet $\mathcal{A}=\{a,b\}$ and the Bernoulli measure $\mu$ such that $\mu(a)=p$ and $\mu(b)=1-p$ with $0<p<1$. Hypothesis (A) can be easily checked and since this process is $\alpha$-mixing with exponential decay, to apply our main theorem, we need to compute the R\'enyi entropy of the pushforward measure.
 
 Let $n\in\N$. We assume that $n$ is even (the odd case can be treated similarly). We observe that by definition of the run-length encoder, cylinders of length $n$ can only have two types, i.e. the cylinder is of type 1 and $C_n=(a,k_1)(b,k_2)(a,k_3)\dots(a,k_{n-1})(b,k_n)$ with $k_1,\dots, k_n\in\N$ or the cylinder is of type 2 and $C_n=(b,k_1)(a,k_2)(b,k_3)\dots(b,k_{n-1})(a,k_n)$ with $k_1,\dots, k_n\in\N$. 
 
It is important to notice that  $f^{-1}\left((a,k_1)(b,k_2)\dots(b,k_n)\right)=\underset{k_1}{\underbrace{a\dots a}}\underset{k_2}{\underbrace{b\dots b}}\dots\underset{k_n}{\underbrace{b\dots b}}a$. Indeed if the last symbol of $C_n$ is $(b,k_n)$, it does not only inform us that in the preimage we have a concatenation of $k_n$ symbols $b$ but also it imposes that this concatenation must be followed by a symbol $a$, otherwise, if it was a symbol b, the last symbol of $C_n$ would not be $(b,k_n)$.
 
 Thus, we have
 \begin{eqnarray*}
 \sum_{C_n} f_*\mu(C_n)^2&=& \sum_{C_n \textrm{ of type 1}} f_*\mu(C_n)^2+\sum_{C_n \textrm{ of type 2}} f_*\mu(C_n)^2\\
 &=&\sum_{k_1,\dots, k_n} \mu(f^{-1}(a,k_1)(b,k_2)\dots(b,k_n))^2+ \sum_{k_1,\dots, k_n\in\N}\mu(f^{-1}(b,k_1)(a,k_2)\dots(a,k_n))^2\\
 &=&\sum_{k_1,\dots, k_n} \mu(\underset{k_1}{\underbrace{a\dots a}}\underset{k_2}{\underbrace{b\dots b}}\dots\underset{k_n}{\underbrace{b\dots b}}a)^2+ \mu(\underset{k_1}{\underbrace{b\dots b}}\underset{k_2}{\underbrace{a\dots a}}\dots\underset{k_n}{\underbrace{a\dots a}}b)^2\\
 &=&\sum_{k_1,\dots, k_n} \left(\mu(a)^{k_1}\mu(b)^{k_2}\dots\mu(b)^{k_n}\mu(a)\right)^2+  \left(\mu(b)^{k_1}\mu(a)^{k_2}\dots\mu(a)^{k_n}\mu(b)\right)^2\\
 &=&\sum_{k_1,\dots, k_n} p^{2k_1}(1-p)^{2k_2}\dots(1-p)^{2k_n}p^2+  (1-p)^{2k_1}p^{2k_2}\dots p^{2k_n}(1-p)^2\\
 &=&p^2\left(\sum_{k=1}^{+\infty} \left(p^{2}\right)^k\right)^{n/2}\left(\sum_{k=1}^{+\infty} \left((1-p)^{2}\right)^k\right)^{n/2}+(1-p)^2\left(\sum_{k=1}^{+\infty} \left(p^{2}\right)^k\right)^{n/2}\left(\sum_{k=1}^{+\infty} \left((1-p)^{2}\right)^k\right)^{n/2}\\
 &=&\left(p^2+(1-p)^2\right)\left(\frac{p^2}{1-p^2}\right)^{n/2}\left(\frac{(1-p)^2}{1-(1-p)^2}\right)^{n/2}.
 \end{eqnarray*}
 This implies that the R\'enyi entropy of the pushforward measure exists and we have
 \[H_2(f_*\mu)=-\displaystyle\underset{n \to \infty}{{\lim}}\frac{1}{n}\log\sum\limits_{C_n} f_*\mu(C_n)^2=-\frac{1}{2}\log\left(\frac{p(1-p)}{(1+p)(2-p)}\right).\]
 Finally, applying Theorem~\ref{thprinc}, we have for almost every realizations $\x,\y$
 \[\lim_{n \to \infty}\frac{\M_n(\x,\y)}{\log n} =\frac{4}{\log\left(\frac{(1+p)(2-p)}{p(1-p)}\right)}.\]
 
 \subsection{Markov chain with two states}\label{exmark1}
 Let us consider the alphabet $\mathcal{A}=\{a,b\}$ and the transition matrix $P=(p_{ij})_{i,j\in\mathcal{A}}$ with $p_{aa}=p$ and $p_{bb}=q$ where $0<p,q<1$. The stationary measure $\mu$ is $\psi$-mixing with exponential decay (see e.g. \cite{brad1}) and hypothesis (A) is satisfied (see e.g. \cite[Lemma 1]{GS}). Thus to apply our theorem we will compute the R\'enyi entropy of the pushforward measure.
 
As in the Bernoulli case, assuming that $n$ is even, cylinders of length $n$ can only have two forms, i.e. $C_n=(a,k_1)(b,k_2)(a,k_3)\dots(a,k_{n-1})(b,k_n)$ or $C_n=(b,k_1)(a,k_2)(b,k_3)\dots(b,k_{n-1})(a,k_n)$ with $k_1,\dots, k_n\in\N$. Thus, we have
 \begin{eqnarray*}
 \sum_{C_n} f_*\mu(C_n)^2
 &=&\sum_{k_1,\dots, k_n} \mu(\underset{k_1}{\underbrace{a\dots a}}\underset{k_2}{\underbrace{b\dots b}}\dots\underset{k_n}{\underbrace{b\dots b}}a)^2+ \mu(\underset{k_1}{\underbrace{b\dots b}}\underset{k_2}{\underbrace{a\dots a}}\dots\underset{k_n}{\underbrace{a\dots a}}b)^2\\
 &=&\sum_{k_1,\dots, k_n} \left(\mu(a)p_{aa}^{k_1-1}p_{ab}p_{bb}^{k_2-1}\dots p_{ab}p_{bb}^{k_n-1}p_{ba}\right)^2+  \left(\mu(b)p_{bb}^{k_1-1}p_{ba}p_{aa}^{k_2-1}\dots p_{ba}p_{aa}^{k_n-1}p_{ab}\right)^2\\
 &=&(\mu(a)^2+\mu(b)^2) \left(p_{ab}^{n/2}\right)^2 \left(p_{ba}^{n/2}\right)^2\left(\sum_{k=1}^{\infty}(p_{aa}^2)^{k-1}\right)^{n/2}\left(\sum_{k=1}^{\infty}(p_{bb}^2)^{k-1}\right)^{n/2}\\
 &=&(\mu(a)^2+\mu(b)^2) p_{ab}^n p_{ba}^n \left(\frac{1}{1-p_{aa}^2}\right)^{n/2}\left(\frac{1}{1-p_{bb}^2}\right)^{n/2}\\
 &=&(\mu(a)^2+\mu(b)^2)\left(\frac{1-p}{1+p}\right)^{n/2}\left(\frac{1-q}{1+q}\right)^{n/2}
 \end{eqnarray*}
and the R\'enyi entropy is
\[H_2(f_*\mu)=-\frac{1}{2}\log\left(\frac{(1-p)(1-q)}{(1+p)(1+q)}\right).\]
Applying Theorem~\ref{thprinc}, we have for almost every realizations $\x,\y$
 \[\lim_{n \to \infty}\frac{\M_n(\x,\y)}{\log n} =\frac{4}{\log\left(\frac{(1+p)(1+q)}{(1-p)(1-q)}\right)}.\]

\subsection{Markov chain with more than 2 states}\label{exmark2}

To study Markov chains with more than two states, we will use another strategy which cannot be used for two states. The idea is that when the original process $X$ is a Markov chain with finite alphabet, the process $f(X)$ is a Markov chain with infinite alphabet. However, when working with only two states, this process is not aperiodic preventing us to compute the R\'enyi entropy using the results of \cite{ciuperca} (which are based on Perron-Frobenius Theorem).


 Let us consider the alphabet $\mathcal{A}=\{\alpha_i\}_{1\leq i\leq N}$ and the transition matrix $P=(p_{ij})_{1\leq i,j\leq N}$ with $0<p_{ij}<1$ for every $1\leq i,j\leq N$. The stationary measure $\mu$ is $\psi$-mixing with exponential decay (see e.g. \cite{brad1}) and hypothesis (A) is satisfied (see e.g. \cite[Lemma 1]{GS}). Thus to apply our theorem we will compute the R\'enyi entropy of the pushforward measure.
 
 First of all, we observe that the process $f(X)$ is also a Markov chain on the alphabet $\mathcal{B}=\{(\alpha,k)\}_{\alpha\in\mathcal{A},k\in\N}$ with transition matrix $Q=(q_{(\alpha_i,k)(\alpha_j,\ell)})_{1\leq i,j\leq N, k,\ell \in\N}$ and initial distribution $\pi=(\pi((\alpha_i,k)))_{1\leq i\leq N, k \in\N}$.
By definition of the run-length encoder, we observe that for all $1\leq i\leq N$ and $k,\ell\in\N$
\[q_{(\alpha_i,k)(\alpha_i,\ell)}=\P\left(f(X)_{n+1}=(\alpha_i,\ell)|f(X)_{n}=(\alpha_i,k)\right)=0.\]
Moreover, for $i\neq j$ and $k,\ell\in\N$, we have
\begin{eqnarray*}
q_{(\alpha_i,k)(\alpha_j,\ell)}&=&\P\left(f(X)_{n+1}=(\alpha_j,\ell)|f(X)_{n}=(\alpha_i,k)\right)\\
&=&\mu(\underset{k}{\underbrace{\alpha_i\dots \alpha_i}}\underset{\ell}{\underbrace{\alpha_j\dots \alpha_j}}\alpha_j^c).\mu(\underset{k}{\underbrace{\alpha_i\dots \alpha_i}}\alpha_i^c)^{-1}
\end{eqnarray*}
where $\alpha^c$ can be any symbol in $\mathcal{A}\setminus\{\alpha\}$. Thus, we obtain
\begin{eqnarray*}
q_{(\alpha_i,k)(\alpha_j,\ell)}&=&\frac{p_{ij}p_{jj}^{\ell-1}(1-p_{jj})}{(1-p_{ii})}.
\end{eqnarray*}
Then, for all $1\leq i\leq N$ and $k\in\N$, we have
\[\pi((\alpha_i,k))=\P(f(X)_1=(\alpha_i,k))=\mu(\underset{k}{\underbrace{\alpha_i\dots \alpha_i}}\alpha_i^c)=\mu(\alpha_i)p_{ii}^{k-1}(1-p_{ii}).\]
Since, $f(X)$ is a Markov chain over a countable alphabet, we will compute ${H}_2(f_*\mathbb{P})$ using Theorem~2 in \cite{ciuperca}:

\begin{theorem}[Theorem 2 \cite{ciuperca}]\label{thciuperca}
Let $Y=(Y_n)_{n\in\N}$ be an irreducible and aperiodic Markov chain with denumerable state space $E$, transition matrix $Q=(q(i,j))_{(i,j)\in E^2}$ and initial distribution $\pi=(\pi(i))_{i\in E}$. If we have
\begin{enumerate}
\item[(1.A)] $\sup_{(i,j)\in E}q{(i,j)}<1$;
\item[(1.B)] there exists $\sigma_0<1$ such that for all $s>\sigma_0$
\[\sup_{i\in E}\sum_{j\in E}q{(i,j)}^s<\infty\]
and
\[\sum_{i\in E}\pi(i)^s<\infty;\]
\item[(1.C)] for all $\eps>0$ and all $s>\sigma_0$, there exists some $A\subset E$ with a finite number of elements, such that
\[\sup_{i\in E}\sum_{j\in\in E\setminus A}q{(i,j)}^s<\eps\]
\end{enumerate}
then for the Markov measure $\nu$ with initial distribution $\pi$ and transition matrix $Q$ and for $k>1$, the R\'enyi entropy of order $k$ exists and 
\[{H}_k(\nu)=-\log \lambda_k\]
where $\lambda_k$ is the largest positive eigenvalue of the matrix $Q_k=\left(\left(q{(i,j)}\right)^k\right)_{(i,j)\in E}$.
\end{theorem}

First, we observe that $f(X)$ is irreducible and aperiodic since $0<p_{ij}<1$ for every $1\leq i,j\leq N$.

We also observe that if the alphabet $\mathcal{A}$ as only two symbols, $f(X)$ is periodic of period 2, thus we cannot apply \cite{ciuperca}.

We will now check the other assumptions to apply their results. Since $0<p_{ij}<1$ for every $1\leq i,j\leq N$, we have
\[\sup_{1\leq i,j\leq N, k,\ell \in\N}q_{(\alpha_i,k)(\alpha_j,\ell)}=\sup_{1\leq i,j\leq N, k,\ell \in\N}\frac{p_{ij}p_{jj}^{\ell-1}(1-p_{jj})}{(1-p_{ii})}
\leq \sup_{1\leq j\leq N}(1-p_{jj})<1\]
and Assumption 1.A is satisfied.

Let $s>0$. We have for any $1\leq i,j\leq N$ and $k\in\N$
\begin{eqnarray*}
\sum_{\ell\in\N}q_{(\alpha_i,k)(\alpha_j,\ell)}^s&=&\sum_{\ell\in\N}\left(\frac{p_{ij}p_{jj}^{\ell-1}(1-p_{jj})}{(1-p_{ii})}\right)^s\\
&=&\left(\frac{p_{ij}(1-p_{jj})}{(1-p_{ii})}\right)^s\frac{1}{1-p_{jj}^s}.
\end{eqnarray*}
Thus,
\begin{eqnarray}
\sup_{1\leq i\leq N, k\in\N}\sum_{1\leq j\leq N,\ell\in\N}q_{(\alpha_i,k)(\alpha_j,\ell)}^s&=&\sup_{1\leq i\leq N}\sum_{1\leq j\leq N}\sum_{\ell\in\N}q_{(\alpha_i,k)(\alpha_j,\ell)}^s\nonumber\\
&=&\sup_{1\leq i\leq N}\sum_{1\leq j\leq N}\left(\frac{p_{ij}(1-p_{jj})}{(1-p_{ii})}\right)^s\frac{1}{1-p_{jj}^s}<+\infty.\label{1B1}
\end{eqnarray}
Moreover, 
\begin{eqnarray}
\sum_{1\leq i\leq N, k\in\N}\pi((\alpha_i,k))^s&=&\sum_{1\leq i\leq N, k\in\N}\left(\mu(\alpha_i)p_{ii}^{k-1}(1-p_{ii})\right)^s\nonumber\\
&=&\sum_{1\leq i\leq N}\frac{\mu(\alpha_i)^s(1-p_{ii})^s}{1-p_{ii}^s}<+\infty.\label{1B2}
\end{eqnarray}
\eqref{1B1} and \eqref{1B2} imply that Assumption 1.B is satisfied.

Let $\eps>0$, $s>0$ and define
\[M=\sup_{1\leq i\leq N}\sum_{1\leq j\leq N}\left(\frac{p_{ij}(1-p_{jj})}{(1-p_{ii})}\right)^s.\]
Since $0<p_{ij}<1$ for every $1\leq i,j\leq N$, it exists $m\in \N$ such that for all $1\leq j\leq N$, we have
\[\frac{p_{jj}^{ms}}{1-p_{jj}^s}<\frac{\eps}{M}.\]
Let $A=\{(\alpha_i,k)\}_{1\leq i\leq N,1\leq k<m}$. We observe that $A$ has a finite number of elements and that
\begin{eqnarray}
\sup_{1\leq i\leq N, k\in\N}\sum_{(\alpha_j,\ell)\in\mathcal{B}\setminus A}q_{(\alpha_i,k)(\alpha_j,\ell)}^s&=&\sup_{1\leq i\leq N}\sum_{1\leq j\leq N}\sum_{\ell=m}^\infty q_{(\alpha_i,k)(\alpha_j,\ell)}^s\nonumber\\
&=&\sup_{1\leq i\leq N}\sum_{1\leq j\leq N}\left(\frac{p_{ij}(1-p_{jj})}{(1-p_{ii})}\right)^s\frac{p_{jj}^{ms}}{1-p_{jj}^s}\nonumber\\
&<&M.\frac{\eps}{M}=\eps.\nonumber
\end{eqnarray}
Thus, Assumption 1.C is satisfied.

Finally, since $f(X)$ satisfies all the assumptions of Theorem~\ref{thciuperca}, ${H}_2(f_*\mu)$ exists and 
\[{H}_2(f_*\mu)=-\log \lambda\]
where $\lambda$ is the largest positive eigenvalue of the matrix $Q_2=\left(\left(q_{(\alpha_i,k)(\alpha_j,\ell)}\right)^2\right)_{1\leq i,j\leq N, k,\ell \in\N}$.

Applying Theorem~\ref{thprinc}, we have for almost every realizations $\x,\y$
 \[\lim_{n \to \infty}\frac{\M_n(\x,\y)}{\log n} =\frac{2}{-\log \lambda}.\]

\section{Proof of the main results}\label{secproof}
To prove our theorems, a natural tool one would like to use is the stationarity of the measure. Since $f_* \mu$ is not stationary, we will need to use the stationarity of $\mu$. In order to do so, we will introduce a new object 
$$
\tilde{M}_n(\x,\y)= \displaystyle \max \left\{k \ : f(\sigma^i\x)_1^{k} = f(\sigma^i\y)_1^{k} \ \mbox{for some} \ 0 \leq i,j \leq n-k\right\}.
$$
First of all, we will explain how $\tilde{M}_n$ and $\M_n$ are related.
\begin{lemma}For every sequences $\x,\y$ and every $n\in\N$
\begin{equation}\label{eqmngeq}
\M_n(\x,\y)\geq \tilde{M}_n(\x,\y)-1.
\end{equation}
Moreover, if $|f(\x_1^n)|\geq u(n)$ and $|f(\y_1^n)|\geq u(n)$ for some $u(n)\in\N$ then
\begin{equation}\label{eqmnleq}
\M_{u(n)}(\x,\y)\leq \tilde{M}_n(\x,\y).
\end{equation}
\end{lemma}
\begin{proof}
For two sequences $\x,\y$, assume that $\tilde{M}_n(\x,\y)=\ell$. Thus, by definition, there exist $0 \leq i,j \leq n-\ell$ such that $f(\sigma^i\x)_1^{\ell} = f(\sigma^i\y)_1^{\ell}$.

Moreover, by definition of the run-length encoder $f$, we observe that the cylinder $f(\sigma^i\x)_2^{\ell}$ always appears at some position in the cylinder $f(\x)_1^n$, more precisely it exists $0\leq i'\leq n-\ell$ such that $f(\sigma^i\x)_2^{\ell}=f(\x)_{i'}^{i'+\ell-2}$. Identically, it exists $0\leq j'\leq n-\ell$ such that $f(\sigma^j\y)_2^{\ell}=f(\y)_{j'}^{j'+\ell-2}$.

Thus $f(\x)_{i'}^{i'+\ell-2}=f(\y)_{j'}^{j'+\ell-2}$ which implies that $\M_n(\x,\y)\geq\ell-1=\tilde{M}_n(\x,\y)-1$ and \eqref{eqmngeq} is proved.

Now, assume that if $|f(\x_1^n)|\geq u(n)$ and $|f(\y_1^n)|\geq u(n)$ for some $u(n)\in\N$. One can observe that by definition of the run-length encoder, we must have $u(n)\leq n$. Let us assume that $\M_{u(n)}(\x,\y)=\ell$. Thus, by definition, there exist $0 \leq i,j \leq u(n)-\ell$ such that $f(\x)_i^{i+\ell-1} = f(\y)_j^{j+\ell-1}$.

Moreover, since $|f(\x_1^n)|\geq u(n)$, it exists $0 \leq i' \leq n-\ell$ such that $f(\x)_i^{i+\ell-1} =f(\sigma^{i'}\x)_1^\ell$. Identically, it exists $0 \leq j' \leq n-\ell$ such that $f(\y)_j^{j+\ell-1} =f(\sigma^{j'}\y)_1^\ell$.

Thus, $f(\sigma^{i'}\x)_1^\ell=f(\sigma^{j'}\y)_1^\ell$ which implies that $\tilde{M}_n(\x,\y)\geq \ell =\M_{u(n)}(\x,\y)$ and \eqref{eqmnleq} is proved.

\end{proof}

\begin{proof}[Proof of Theorem~\ref{thinf}]

The proof of this theorem follows in part the lines of the proof of the Theorem 7 in \cite{BaLiRo}, nevertheless some subtle modifications are necessary since $f_* \mu$ is not stationary.

Let $\eps >0$ and denote
$$k_n=\left\lceil\frac{2 \log n + \log \log n}{\underline{H_2}(f_*\mathbb{P})-\eps}\right\rceil.$$
We define $u(n)=\frac{n}{\log n^\delta}$ with $\delta>\frac{1}{h}$ and 
\[V_n=\{\x, |f(\x_1^n)|\geq u(n)\}.\]
First of all, using \eqref{eqmnleq}, we observe that  
\begin{eqnarray}
\P(\M_{u(n)}\geq k_n)&\leq& \P((V_n\times V_n)\cap \M_{u(n)}\geq k_n)+\P((V_n\times V_n)^c)\nonumber\\
&\leq& \P(\tilde{M}_n\geq k_n)+3\mu(V_n^c).\label{ineqmnvnc}
\end{eqnarray}

Firstly, we will estimate $\P(\tilde{M}_n\geq k_n)$. For $0\leq i,j\leq n-1$ we define the following event
$$A_{i,j}=\{f(\sigma^i X)_{1}^{k_n}=f(\sigma^j Y)_{1}^{k_n}\}$$
and the following random variable
\begin{equation}\label{defsn}
S_n= \sum_{i,j= 0, \ldots, n-1}\mathbbm{1}_{A_{i,j}}.
\end{equation}

It follows from our definitions and Markov's inequality that
$$\mathbb{P} \lt(\tilde M_n \geq k_n\rt)=\mathbb{P} \lt(S_n \geq1\rt)  \leq  \E\lt(S_n\rt).$$
Moreover, we use the stationarity of $\mu$ to get
\begin{eqnarray}
\E\lt(S_n\rt)&=&\sum_{0 \leq i,j \leq n-1}\sum_{\omega\in\mathcal{B}^{k_n}}\P\left(f(\sigma^iX)_{1}^{k_n}=f(\sigma^j Y)_{1}^{k_n}=\omega\right)\nonumber\\
&=&\sum_{0 \leq i,j \leq n-1}\sum_{\omega\in\mathcal{B}^{k_n}}\mu\left(f(\sigma^iX)_{1}^{k_n}=\omega\right)\mu\left(f(\sigma^j Y)_{1}^{k_n}=\omega\right)\nonumber\\
&=&n^2\sum_{\omega\in\mathcal{B}^{k_n}}\mu\left(f^{-1}\omega\right)^2=n^2\sum_{\omega\in\mathcal{B}^{k_n}}f_*\mu\left(\omega\right)^2.\label{eqesn}
\end{eqnarray}

Thus, for $n$ large enough, by definition of $\underline{H}_2(f_*\mathbb{P})$ and of $k_n$, we have
\begin{equation}\label{eqmntilde}
\mathbb{P} \lt(\tilde M_n \geq k_n\rt)\leq n^2\sum_{\omega\in\mathcal{B}^{k_n}}f_*\mu\left(\omega\right)^2\leq n^2 e^{-k_n (\underline{H}_2(f_*\mathbb{P})-\eps)} \leq\frac{1}{\log n} .
\end{equation}

Let us now estimate $\mu(V_n^c)=\{\x, |f(\x_1^n)|< u(n)\}$. Let $\x\in V_n^c$, by definition of the run-length encoder, one can notice that since $|f(\x_1^n)|< u(n)$, it exists $a\in\mathcal{A}$ such that the cylinder $a...a$ of length $\left\lceil\frac{n}{u(n)}\right\rceil$ appears at some position in the cylinder $\x_1^n$, more precisely it exists $0\leq i\leq n-\left\lceil\frac{n}{u(n)}\right\rceil$ such that $a...a=\x_{i+1}^{i+\lceil{n}/{u(n)}\rceil}$. Thus, we obtain
\begin{eqnarray}
\mu(V_n^c)&\leq&\mu\left(\bigcup_{a\in\mathcal{A}}\bigcup_{0\leq i\leq n-\left\lceil\frac{n}{u(n)}\right\rceil}\sigma^{-i}a\dots a\right)\label{ineqvnca}\\
&\leq&\sum_{a\in\mathcal{A}}\sum_{0\leq i\leq n-\left\lceil\frac{n}{u(n)}\right\rceil}\mu\left(\sigma^{-i}a\dots a\right)\nonumber\\
&=&\left(n-\left\lceil\frac{n}{u(n)}\right\rceil+1\right)\sum_{a\in\mathcal{A}}\mu\left(a\dots a\right).\nonumber
\end{eqnarray}
Thus, using assumption (A) and since $u(n)=\frac{n}{\log n^\delta}$ with $\delta>\frac{1}{h}$, we obtain
\begin{equation}\label{eqvnc}
\mu(V_n^c)\leq c |\mathcal{A}| n e^{-h n/u(n)}\leq c|\mathcal{A}|\frac{1}{n}
\end{equation}
where $|\mathcal{A}|$ denotes the cardinality of $\mathcal{A}$.

Thus, combining \eqref{ineqmnvnc}, \eqref{eqmntilde} and \eqref{eqvnc}, we obtain
\[\P(\M_{u(n)}\geq k_n)\leq\mathcal{O}\left((\log n)^{-1}\right).\]

Choosing a subsequence $(n_{\ell})_{\ell \in \mathbb{N}}$ such that $n_{\ell}= \lceil e^{\ell^2}\rceil$ we have that $\sum_\ell\mathbb{P} \lt(\M_{u(n_\ell)} \geq k_{n\ell}\rt)<+\infty$. Thus, by the Borel-Cantelli lemma, we have almost surely, if $\ell$ is large enough,

$$\M_{u(n_{\ell})} < k_{n_{\ell}}$$
and  then
\begin{eqnarray*}\label{desigbc}
\frac{ \M_{u(n_{\ell})}}{\log n_{\ell}} \leq  \frac{1}{\underline{H}_2(f_*\mu)-\eps}\lt( 2 + \frac{1+\log \log n_{\ell}}{\log n_{\ell}}\rt).
\end{eqnarray*}

Taking the limit superior in this inequality and observing that $(\M_{n})_n$, $(u(n))_n$ and $(n_\ell)_\ell$ are increasing, that $\lim\limits_{\ell \to \infty} \dfrac{\log n_{\ell}}{ \log n_{\ell+1}}=1$ and that $\lim\limits_{n \to \infty} \dfrac{\log u(n)}{ \log n}=1$, we obtain almost surely

$$
\underset{n \rightarrow \infty}{\underline{\lim}} \frac{\M_n}{\log n} = \underset{\ell \rightarrow \infty}{\underline{\lim}} \frac{ \M_{u(n)}}{\log n} = \underset{\ell \rightarrow \infty}{\underline{\lim}} \frac{ \M_{u(n_\ell)}}{\log n_{\ell}} \leq \frac{2}{\underline{H}_{2}(f_*\mathbb{P})-\eps}.
$$
And the theorem is proved since $\eps$ can be chosen arbitrarily small.

\end{proof}

 \begin{proof}[Proof of Theorem~\ref{thprinc}]
 We will proved here the theorem when the process is $\alpha$-mixing with exponential decay, the $\psi$-mixing case can be obtained similarly by a slight modification. Without loss of generality, we will also assume that $\alpha(g)=e^{-g}$.
 
Let $\eps>0$ and define
$$k_n=\left\lfloor\frac{2 \log n + b \log \log n}{\overline{H_2}(f_*\mu)+\eps}\right\rfloor$$
where $b$ is a constant to be chosen.

First of all, using \eqref{eqmngeq}  and \eqref{defsn}, we observe that
\[\mathbb{P}(\M_n<k_n-1)\leq \mathbb{P}(\tilde{M}_n<k_n)=\mathbb{P}(S_n=0)\]
thus, by Chebyshev's inequality
\begin{equation}\label{eqcheb}
\mathbb{P} \lt(\M_n < k_n-1\rt) \leq  \frac{\textrm{var} \lt(S_n\rt)}{\E\lt(S_n\rt)^2}.
\end{equation}

For the variance of $S_n$, we use the following lemma (which will be proved after the proof of the theorem):
\begin{lemma}\label{lemvar}
Under the assumptions of Theorem~\ref{thprinc}, for $g\in \N$, we have
\begin{eqnarray*}
\var(S_n)&\leq& 4(g+k_n)\E(S_n)^{3/2}+4(g+k_n)^2\E(S_n)+2n^4\alpha(g+k_n-k_n^3)\\
& &+4n^3(g+k_n)\alpha(g+k_n-k_n^3)
+\left(2n^4+4n^3(g+k_n)\right)ce^{-hk_n^2}.
\end{eqnarray*}
\end{lemma}

Since $k_n=\mathcal{O}(\log n)$, for $\beta>3$ and $g=(\log n)^\beta$ we have
\[(g+k_n-k_n^3)\sim(\log n)^{\beta}.\]
 Thus, since $\alpha(g)=e^{-g}$ and since $\log (n^5)=o\left((\log n)^\beta\right)$, we obtain

\begin{equation}\label{eqvar2}
2n^4 \alpha(g+k_n-k_n^2)=\mathcal{O}(n^{-1})
\end{equation}
and
\begin{equation}\label{eqvar3}
4n^3(g+k_n)\alpha(g+k_n-k_n^2)=\mathcal{O}(n^{-1}).
\end{equation}
By definition of $k_n$, for $n$ large enough we have $hk_n^2\geq 5\log n$, thus we obtain
\begin{equation}\label{eqvar1}
\left(2n^4+4n^3(g+k_n)\right)ce^{-hk_n^2}=\mathcal{O}(n^{-1}).
\end{equation}

Thus, combining \eqref{eqcheb} together with Lemma~\ref{lemvar}, \eqref{eqvar2},  \eqref{eqvar3} and \eqref{eqvar1}, we have
$$\mathbb{P} \lt(\M_n < k_n-1\rt) \leq  \frac{\var \lt(S_n\rt)}{\E\lt(S_n\rt)^2}\leq \frac{4(g+k_n)}{\E\lt(S_n\rt)^{1/2}}+\frac{4(g+k_n)^2}{\E\lt(S_n\rt)}+\mathcal{O}(n^{-1}).$$
By \eqref{eqesn} and by definitions of $k_n$ and the R\'enyi entropy, we have
\[\E\lt(S_n\rt)=n^2\sum_{\omega\in\mathcal{B}^{k_n}}f_*\mu\left(\omega\right)^2\geq n^2 e^{-k_n(\underline{H}_2(f_*\mu)+\eps)}\geq (\log n)^{-b}\]
and recalling that $(g+k_n)\sim(\log n)^{\beta}$, one can choose $b\ll -1$ to obtain
\[\mathbb{P} \lt(\M_n < k_n-1\rt) \leq \mathcal{O}\left((\log n)^{-1}\right).\]

Choosing a subsequence $(n_{\ell})_{\ell \in \mathbb{N}}$ such that $n_{\ell}= \lceil e^{\ell^2}\rceil$ we have that $\sum_\ell\mathbb{P} \lt(\M_{n_\ell} < k_{n\ell}-1\rt)<+\infty$. Thus, by the Borel-Cantelli lemma, we have almost surely, if $\ell$ is large enough, 
$$\M_{n_\ell} \geq k_{n\ell}-1$$
and  then
\begin{eqnarray*}\label{desigbc}
\frac{ \M_{n_\ell}}{\log n_{\ell}} \geq  \frac{1}{\overline{H}_2(f_*\mu)+\eps}\lt( 2 + b\frac{1+\log \log n_{\ell}}{\log n_{\ell}}\rt)-\frac{1}{\log n_{\ell}}.
\end{eqnarray*}
Taking the limit inferior in this inequality and observing that $(\M_{n})_n$ and $(n_\ell)_\ell$ are increasing and that $\lim\limits_{\ell \to \infty} \dfrac{\log n_{\ell}}{ \log n_{\ell+1}}=1$, we obtain almost surely
$$
\underset{n \rightarrow \infty}{\underline{\lim}} \frac{\M_n}{\log n} = \underset{\ell \rightarrow \infty}{\underline{\lim}} \frac{ \M_{n_\ell}}{\log n_{\ell}} \geq \frac{2}{\overline{H}_{2}(f_*\mu)+\eps}.
$$
And the theorem is proved since $\eps$ can be chosen arbitrarily small.

\end{proof}

\begin{proof}[Proof of Lemma~\ref{lemvar}]
To estimate the variance of $S_n$, we observe that
\begin{eqnarray}\label{vardis}
\var\lt(S_n\rt) &  =  & \sum_{0 \leq i,i',j,j' \leq n-1}\E\left(\mathbbm{1}_{A_{i,j}}\mathbbm{1}_{A_{i',j'}}\right)-\E\lt(S_n\rt)^2.
\end{eqnarray}

Let $g\in\N$. Firstly, we assume that $i'-i > g+ k_n$ and $j'-j > g+ k_n$ (the case $i-i' > g+ k_n$ and $j-j' > g+ k_n$ can be treated identically), then we have
\begin{eqnarray*}
\E\left(\mathbbm{1}_{A_{i,j}}\mathbbm{1}_{A_{i',j'}}\right)&=&\sum_{\omega,\omega'\in\mathcal{B}^{k_n}}\P\left(f(\sigma^iX)_{1}^{k_n}=f(\sigma^jY)_{1}^{k_n}=\omega, f(\sigma^{i'}X)_{1}^{k_n}=f(\sigma^{j'}Y)_{1}^{k_n}=\omega'\right)\\
&=&\sum_{\omega,\omega'\in\mathcal{B}^{k_n}}\mu\left(f(\sigma^iX)_{1}^{k_n}=\omega, f(\sigma^{i'}X)_{1}^{k_n}=\omega'\right)\mu\left(f(\sigma^jY)_{1}^{k_n}=\omega, f(\sigma^{j'}Y)_{1}^{k_n}=\omega'\right)\\
&=&\sum_{\omega,\omega'\in\mathcal{B}^{k_n}}\mu\left(f(X)_{1}^{k_n}=\omega, f(\sigma^{i'-i}X)_{1}^{k_n}=\omega'\right)\mu\left(f(Y)_{1}^{k_n}=\omega, f(\sigma^{j'-j}Y)_{1}^{k_n}=\omega'\right)
\end{eqnarray*}
by stationarity of $\mu$.

To use the mixing property, we need to work with cylinders whose preimage under $f$ does not have a length too large so that the gap is preserved. Thus, we define the set
\[\mathcal{Z}_n=\left\{\omega\in \mathcal{B}^{k_n}: |f^{-1}\omega|\leq k_n^3\right\}\]
and, using the $\alpha$-mixing, we obtain

\begin{eqnarray}
& & \sum_{\substack{\omega\in\mathcal{Z}_n \\ \omega'\in \mathcal{B}^{k_n} }}\mu\left(f(X)_{1}^{k_n}=\omega, f(\sigma^{i'-i}X)_{1}^{k_n}=\omega'\right)\mu\left(f(Y)_{1}^{k_n}=\omega, f(\sigma^{j'-j}Y)_{1}^{k_n}=\omega'\right)\nonumber\\
&\leq& \sum_{\substack{\omega\in\mathcal{Z}_n \\ \omega'\in \mathcal{B}^{k_n} }}\left[\mu\left(f(X)_{1}^{k_n}=\omega\right)\mu\left(f(X)_{1}^{k_n}=\omega'\right)+\alpha(g+k_n-k_n^3)\right]\nonumber\\
& &\times \mu\left(f(Y)_{1}^{k_n}=\omega, f(\sigma^{j'-j}Y)_{1}^{k_n}=\omega'\right)\nonumber\\
&\leq& \sum_{\substack{\omega\in\mathcal{Z}_n \\ \omega'\in \mathcal{B}^{k_n} }}\mu\left(f(X)_{1}^{k_n}=\omega\right)\mu\left(f(X)_{1}^{k_n}=\omega'\right)\left[\mu\left(f(Y)_{1}^{k_n}=\omega\right)\mu\left(f(Y)_{1}^{k_n}=\omega'\right)+\alpha(g+k_n-k_n^3)\right]\nonumber\\
& &+\sum_{\substack{\omega'\in \mathcal{B}^{k_n} }}\alpha(g+k_n-k_n^3)\mu\left(f(\sigma^{j'-j}Y)_{1}^{k_n}=\omega'\right)\nonumber\\
&\leq&2\alpha(g+k_n-k_n^3)+ \left[\sum_{\omega\in\mathcal{B}^{k_n}}f_*\mu\left(\omega\right)^2\right]^2. \label{ineqijzn}
\end{eqnarray}
When the length of the preimage cylinders is too large (i.e. $\omega\notin \mathcal{Z}_n$), we cannot use the mixing property, however, we observe that
\begin{eqnarray*}
& & \sum_{\substack{\omega\in\mathcal{Z}_n^C \\ \omega'\in \mathcal{B}^{k_n} }}\mu\left(f(X)_{1}^{k_n}=\omega, f(\sigma^{i'-i}X)_{1}^{k_n}=\omega'\right)\mu\left(f(Y)_{1}^{k_n}=\omega, f(\sigma^{j'-j}Y)_{1}^{k_n}=\omega'\right)\nonumber\\
&\leq& \sum_{\omega\in\mathcal{Z}_n^C}\mu\left(f(X)_{1}^{k_n}=\omega\right)\sum_{ \omega'\in \mathcal{B}^{k_n} }\mu\left(f(Y)_{1}^{k_n}=\omega, f(\sigma^{j'-j}Y)_{1}^{k_n}=\omega'\right)\nonumber\\
&\leq& \sum_{\omega\in\mathcal{Z}_n^C}\mu\left(f(X)_{1}^{k_n}=\omega\right)\mu\left(f(Y)_{1}^{k_n}=\omega\right)\nonumber\\
&=& \sum_{\omega\in\mathcal{Z}_n^C}\mu\left(f^{-1}\omega\right)^2.\nonumber\\
\end{eqnarray*}

Using the same argument as the one leading to \eqref{ineqvnca}, we observe that if $\omega\in \mathcal{Z}_n^C$ then $|\omega|=k_n$ and $|f^{-1}\omega|> k_n^3$, thus there exist $a\in \mathcal{A}$ and $\kappa\in\N$ such that $f^{-1}\omega \subset\sigma^{-\kappa}a\dots a$ where $|a\dots a|\geq\frac{k_n^3}{k_n}=k_n^2$. Thus, one can use assumption (A) to observe that $\mu(f^{-1}\omega)\leq\mu(\sigma^{-\kappa}a\dots a)=\mu(a\dots a)\leq ce^{-hk_n^2}$. Since this inequality is valid for any $\omega \in Z_n^C$, we obtain
\begin{eqnarray}
& & \sum_{\substack{\omega\in\mathcal{Z}_n^C \\ \omega'\in \mathcal{B}^{k_n} }}\mu\left(f(X)_{1}^{k_n}=\omega, f(\sigma^{i'-i}X)_{1}^{k_n}=\omega'\right)\mu\left(f(Y)_{1}^{k_n}=\omega, f(\sigma^{j'-j}Y)_{1}^{k_n}=\omega'\right)\nonumber\\
&\leq& \sum_{\omega\in\mathcal{Z}_n^C}\mu\left(f^{-1}\omega\right)^2\nonumber\\
&\leq&ce^{-hk_n^2}\sum_{\omega\in\mathcal{Z}_n^C}\mu\left(f^{-1}\omega\right)\nonumber\\
&\leq&ce^{-hk_n^2}\label{ineqnotinz}.
\end{eqnarray}

Thus, \eqref{ineqijzn} together with \eqref{ineqnotinz} gives us that when $i'-i > g+ k_n$ and $j'-j > g+ k_n$
\begin{equation*}\label{grandgrand1}
\E\left(\mathbbm{1}_{A_{i,j}}\mathbbm{1}_{A_{i',j'}}\right)\leq 2\alpha(g+k_n-k_n^3)+ \left[\sum_{\omega\in\mathcal{B}^{k_n}}f_*\mu\left(\omega\right)^2\right]^2+ce^{-hk_n^2}.
\end{equation*}
We observe that when $i'-i > g+ k_n$ and $j-j' > g+ k_n$ (the case $i-i' > g+ k_n$ and $j'-j > g+ k_n$ can be treated identically) then we can obtain \eqref{ineqijzn} only if we restrict our sum to $\omega'\in\mathcal{Z}_n$, thus the estimate for $\E\left(\mathbbm{1}_{A_{i,j}}\mathbbm{1}_{A_{i',j'}}\right)$ will be slightly different and we will have
\begin{equation*}\label{grandgrand2}\E\left(\mathbbm{1}_{A_{i,j}}\mathbbm{1}_{A_{i',j'}}\right)\leq 2\alpha(g+k_n-k_n^3)+ \left[\sum_{\omega\in\mathcal{B}^{k_n}}f_*\mu\left(\omega\right)^2\right]^2+2ce^{-hk_n^2}.
\end{equation*}
Thus, since $\card\{0\leq i,j,i',j'\leq n-1 \textrm{ s.t. $|i'-i| > g+ k_n$ and $|j'-j| > g+ k_n$}\}\leq n^4$, we have
\begin{eqnarray}\label{grandgrand2-b}
\sum_{\substack{|i'-i| > g+ k_n \\ |j'-j| > g+ k_n}}\E\left(\mathbbm{1}_{A_{i,j}}\mathbbm{1}_{A_{i',j'}}\right)
&\leq& n^4\left(2\alpha(g+k_n-k_n^3)+ \left[\sum_{\omega\in\mathcal{B}^{k_n}}f_*\mu\left(\omega\right)^2\right]^2+2ce^{-hk_n^2}\right)\nonumber\\
&=&n^4\left(2\alpha(g+k_n-k_n^3)+2ce^{-hk_n^2}\right)+ \E(S_n)^2 .\label{grandgrand2-b}
\end{eqnarray}
Now, we assume that $i'-i > g+ k_n$ and $0\leq j'-j \leq g+ k_n$ (the other cases such that $\lt|i'-i\rt| > g+ k_n$ and $\lt|j'-j\rt| \leq g+ k_n$  or such that $\lt|i'-i\rt| \leq g+ k_n$ and $\lt|j'-j\rt| > g+ k_n$ can be treated identically), using the mixing property as in \eqref{ineqijzn} and then, using H\"older's inequality, we have
\begin{eqnarray}
& & \sum_{\substack{\omega\in\mathcal{Z}_n \\ \omega'\in \mathcal{B}^{k_n} }}\mu\left(f(X)_{1}^{k_n}=\omega, f(\sigma^{i'-i}X)_{1}^{k_n}=\omega'\right)\mu\left(f(Y)_{1}^{k_n}=\omega, f(\sigma^{j'-j}Y)_{1}^{k_n}=\omega'\right)\nonumber\\
&\leq& \alpha(g+k_n-k_n^3)+\sum_{\substack{\omega\in\mathcal{Z}_n \\ \omega'\in \mathcal{B}^{k_n} }}\mu\left(f(X)_{1}^{k_n}=\omega\right)\mu\left(f(X)_{1}^{k_n}=\omega'\right) \mu\left(f(Y)_{1}^{k_n}=\omega, f(\sigma^{j'-j}Y)_{1}^{k_n}=\omega'\right)\nonumber\\
&\leq&\alpha(g+k_n-k_n^3)+\int_{\mathcal{A}^\N}\mu\left(f(X)_{1}^{k_n}=f(y)_{1}^{k_n}\right)\mu\left(f(\sigma^{j'-j}X)_{1}^{k_n}=f(\sigma^{j'-j}y)_{1}^{k_n}\right)d\mu(y)\nonumber\\
&\leq&\alpha(g+k_n-k_n^3)\nonumber\\
& &+\left[\int_{\mathcal{A}^\N}\mu\left(f(X)_{1}^{k_n}=f(y)_{1}^{k_n}\right)^2d\mu(y)\right]^{1/2}\left[\int_{\mathcal{A}^\N}\mu\left(f(\sigma^{j'-j}X)_{1}^{k_n}=f(\sigma^{j'-j}y)_{1}^{k_n}\right)^2d\mu(y)\right]^{1/2}\nonumber\\
&=&\alpha(g+k_n-k_n^3)+\sum_{\omega\in\mathcal{B}^{k_n}}f_*\mu\left(\omega\right)^3\nonumber\\
&\leq&\alpha(g+k_n-k_n^3)+\left[\sum_{\omega\in\mathcal{B}^{k_n}}f_*\mu\left(\omega\right)^2\right]^{3/2} \nonumber
\end{eqnarray}
where the last inequality comes from the subaditivity of the map $x\mapsto x^{3/2}$.

For the terms with $\omega\notin \mathcal{Z}_n$ we will use the estimate \eqref{ineqnotinz}. Thus, for $\lt|i'-i\rt| > g+ k_n$ and $\lt|j'-j\rt| \leq g+ k_n$ we have
\begin{equation*}\label{grandpetit}
\E\left(\mathbbm{1}_{A_{i,j}}\mathbbm{1}_{A_{i',j'}}\right)\leq \alpha(g+k_n-k_n^3)+\left[\sum_{\omega\in\mathcal{B}^{k_n}}f_*\mu\left(\omega\right)^2\right]^{3/2}+ce^{-hk_n^2}.
\end{equation*}
Moreover, since $\card\{0\leq i,j,i',j'\leq n-1 \textrm{ s.t. $|i'-i| > g+ k_n$ and $|j'-j| \leq g+ k_n$}\}\leq 2n^3(g+k_n)$, we have
\begin{eqnarray}
& &\sum_{\substack{|i'-i| > g+ k_n \\ |j'-j| \leq g+ k_n}}\E\left(\mathbbm{1}_{A_{i,j}}\mathbbm{1}_{A_{i',j'}}\right)+\sum_{\substack{|i'-i| \leq g+ k_n \\ |j'-j| > g+ k_n}}\E\left(\mathbbm{1}_{A_{i,j}}\mathbbm{1}_{A_{i',j'}}\right)\nonumber\\
&\leq& 4n^3(g+k_n)\left(\alpha(g+k_n-k_n^3)+\left[\sum_{\omega\in\mathcal{B}^{k_n}}f_*\mu\left(\omega\right)^2\right]^{3/2}+ce^{-hk_n^2}\right)\nonumber\\
&=&4n^3(g+k_n)\left(\alpha(g+k_n-k_n^3)+ce^{-hk_n^2}\right)+4(g+k_n)\E(S_n)^{3/2}.\label{grandpetit2}\end{eqnarray}
Finally, when $\lt|i'-i\rt| \leq g+ k_n$ and $\lt|j'-j\rt| \leq g+ k_n$, we just observe that
\begin{eqnarray}
\E\left(\mathbbm{1}_{A_{i,j}}\mathbbm{1}_{A_{i',j'}}\right)&\leq&\E\left(\mathbbm{1}_{A_{i,j}}\right)\nonumber\\
&=&\sum_{\omega\in\mathcal{B}^{k_n}}f_*\mu\left(\omega\right)^2\nonumber
\end{eqnarray}
and since $\card\{0\leq i,j,i',j'\leq n-1 \textrm{ s.t. $|i'-i| \leq g+ k_n$ and $|j'-j| \leq g+ k_n$}\}\leq 4n^2(g+k_n)^2$, we have
\begin{equation}\label{petitpetit}\sum_{\substack{|i'-i| \leq g+ k_n \\ |j'-j| \leq g+ k_n}}\E\left(\mathbbm{1}_{A_{i,j}}\mathbbm{1}_{A_{i',j'}}\right)\leq 4n^2(g+k_n)^2\sum_{\omega\in\mathcal{B}^{k_n}}f_*\mu\left(\omega\right)^2=4(g+k_n)^2\E(S_n).\end{equation}
Combining together the estimates \eqref{vardis}, \eqref{grandgrand2-b}, \eqref{grandpetit2} and \eqref{petitpetit}, we obtain
\begin{eqnarray*}
\var(S_n)&\leq&
4(g+k_n)\E(S_n)^{3/2}+4(g+k_n)^2\E(S_n)+2n^4\alpha(g+k_n-k_n^3)\\
& &+4n^3(g+k_n)\alpha(g+k_n-k_n^3)
+\left(2n^4+4n^3(g+k_n)\right)ce^{-hk_n^2}
\end{eqnarray*}
and the lemma is proved.
\end{proof}



\subsection*{Acknowledgements}The author would like to thank Rodrigo Lambert for useful discussions and the referee for useful suggestions and corrections to improve the paper.


\end{document}